\documentclass[12pt]{article}
\usepackage{hyperref}
\usepackage{tikz}
\usepackage[active]{srcltx}
\usepackage[curve]{xypic}
\usepackage{graphicx}

\usepackage[active]{srcltx}
\usepackage{epsfig,amsmath}

\usepackage[english]{babel}
\usepackage{amsmath,amsfonts,amssymb,amscd,color,epsfig,amsthm}

\usepackage{amsmath,amssymb,amscd,amsthm,amsfonts,mathrsfs}

\newtheorem{newthm}{Theorem}
\newtheorem{theorem}{Theorem}[section]
\newtheorem{lemma}[theorem]{Lemma}
\newtheorem{proposition}[theorem]{Proposition}
\newtheorem{corollary}[theorem]{Corollary}

\theoremstyle{remark}

\theoremstyle{plain}

\numberwithin{equation}{section}

\def\MMM{{\cal M}}

\def\g{\gamma}

\def\R{\mbox{$\mathbb R$}}

\def\C{\mbox{$\mathbb C$}}
\def\T{\mbox{$\mathbb T$}}
\def\D{\mathbb D}
\def\Z{\mbox{$\mathbb Z$}}

\def\lv{ \left(\begin{matrix} }
 \def\rv{\end{matrix}\right)}

\def\g{\gamma}

\def\cal{\mathcal}

\def\dw{{\dw}}

\newcommand{\mylabel}[1]{\label{#1}}

\newcommand{\REFEQN}[1] { \begin{equation}\mylabel{#1} }
\newcommand{\ENDEQN}{\end{equation}}
\newcommand{\REFTHM}[1] { \begin{theorem}\mylabel{#1} }
\newcommand{\ENDTHM}{\end{theorem}}
\newcommand{\REFNTH}[1] { \begin{newthm}\mylabel{#1} }
\newcommand{\ENDNTH}{\end{newthm}}
\newcommand{\REFPROP}[1]{\begin{proposition}\mylabel{#1} }
\newcommand{\ENDPROP}{\end{proposition} }
\newcommand{\REFLEM}[1]{\begin{lemma}\mylabel{#1} }
\newcommand{\ENDLEM}{\end{lemma} }
\newcommand{\REFCOR}[1]{\begin{corollary}\mylabel{#1} }
\newcommand{\ENDCOR}{\end{corollary} }

\def\ov{\overline}

\def\T{{\mathbb T}}

\author{
Yan Gao\footnote{supported by the grant no. 11501383 of NSFC.}
\and
Jinsong Zeng\footnote{the corresponding author.}
}

\date{\today}

\begin{document}

\title{Non-recurrent parameter rays of the Mandelbrot set\footnote{\emph{Keywords: }Mandelbrot sets; non-recurrent. \ \ \emph{AMS[2010]Classification:} 37F45 }}

\maketitle
\begin{abstract}
In this paper, we prove that any parameter ray at a non-recurrent angle $\theta$ lands at a non-recurrent parameter $c$ with $\theta$ a characteristic angle of $f_c$; and conversely, every non-recurrent parameter $c$ is the landing point of one or two parameter rays at non-recurrent angles, and these angles are exactly the characteristic angles of $f_c$.
%\vskip 0.2cm

\end{abstract}

\section{Introduction}
  The quadratic family $\{f_c:z\mapsto z^2+c\}$ exhibits rich dynamics, when iterated.  The Mandelbrot set
\[\MMM:=\{c\in\C\mid f_c^n(c)\not\to\infty\text{ as }n\to\infty\}\]
organizes the space of quadratic polynomials up to conjugacy and has a beautiful structure.
It has been a very active area of research in the past few decades.
The importance of the Mandelbrot set is due to the fact that it is the simplest non-trival parameter space of analytic families of iterated holomorphic maps, and because of its universality as explained in \cite{DH1, Mc}

Much of the topological and combinatorial structures of the Mandelbrot set has been discovered by the work of Douady and Hubbard \cite{DH2}. A fundamental result in \cite{DH2} is to describe the landing behavior of the rational parameter rays.  One can see also \cite{Mil2,S,PR} for alternative approaches.

In this article, we study the landing property of irrational, precisely the \emph{non-recurrent}, parameter rays.

Set $\T:=\R/\Z$ and $\tau:\T\to\T$ the angle doubling map, i.e., $\tau(\theta)=2\theta\ {\rm mod}\Z,\theta\in\T$.  By abuse of notations, we  identify $\T$ with the unit circle under the correspondence $t\mapsto e^{2\pi it}$.

An angle $\theta\in\T$ is called \emph{non-recurrent} if  $\tau^n(\theta)\not=\tau^m(\theta)$ for any $0\leq m<n$, and there exists  $\delta>0$ such that $|\theta-\tau^n(\theta)|>\delta$ for all $n\geq 1$. A quadratic polynomial $f_c$, or the parameter $c$, is called \emph{non-recurrent} if $c\in \MMM$, all periodic points of $f_c$ are repelling and $|f^n_c(0)|>\delta$ for all $n\geq1$ and  a positive constance $\delta$.

It is known that the Julia sets of non-recurrent quadratic polynomials are connected and locally-connected. The angle of an external ray landing at the critical value $c$ is said to be a \emph{characteristic angle} of $f_c$. Refer to Section \ref{preliminary} for the definitions of  external rays and parameter rays. The following is our main theorem.

\begin{theorem}\label{main}
Any parameter ray at a non-recurrent angle lands at a non-recurrent parameter $c$ such that $\theta$ is a characteristic angle of $f_c$. Conversely, every non-recurrent parameter $c$ is the landing point of one or two parameter rays at non-recurrent angles, and these angles are exactly the characteristic angles of $f_c$.
\end{theorem}

 Our proof is based on Kiwi's Combinatorial Continuity Theorem \cite[Theorem 1]{Ki2} and Yoccoz Rigidity Theorem \cite[Theorem III]{H} (or \cite[Theorem 4.1]{Ze}). We will review some background of polynomial dynamics and fix  notations in Section \ref{preliminary}, and introduce the concept of \emph{real lamination} in Section \ref{p-l}. In Section \ref{key-lemma}, we verify two combinatorial results used in the proof of the main theorem, and the proof of Theorem \ref{main} is left in Section \ref{proving}.

\section{Polynomial dynamics}\label{preliminary}
One can refer to \cite{DH2} for the details of the content in this section.

Let $f_c(z)=z^2+c,z\in\C$ be a quadratic polynomial. The set of all points
which remain bounded under all iterations of $f_c$ is called the \emph{Filled-in Julia set} $K_c$. The boundary of the Filled-in Julia set is defined to be the \emph{Julia set} $J_c$ and the complement of the Julia set is defined to be its \emph{Fatou set} $F_c$.

If $c\in \MMM$, the filled-in Julia set $K_c$ is \emph{simply-connected}, i.e., $\C\setminus K_c$ is connected. There is then a unique biholomorphic map $\phi_c$ from $\C\setminus K_c$ onto $\C\setminus\ov{\D}$, called the \emph{B\"{o}ettcher coordinate}, such that $\lim_{z\to\infty}\phi_c(z)/z=1$ and $\phi_c\circ f_c(z)=\phi_c(z)^2$ for $z\in\C\setminus K_f$.
 The preimage of $(1,\infty)e^{2\pi it}$ under $\phi_c$, denoted $R_c(t)$, is called  the {\it external ray at angle} $t$. We say that the external ray $R_c(t)$ \emph{lands} if $\ov{R_c(t)}\cap K_c$ is a singleton, and this point, denoted by $\g_c(t)$, is called the \emph{landing point} of $R_c(t)$.

In the parameter plane, the Mandelbrot set $\MMM$ is simply-connected, and there is a biholomorphic map $\Phi$ from $\C\setminus \MMM$ onto $\C\setminus\ov{\D}$. The \emph{parameter ray at angle $\theta$} is defined as the set $R_\MMM(\theta):=\{\Phi^{-1}(re^{2\pi i\theta}),r>1\}$. Similarly, if $\ov{R_\MMM(\theta)}\cap \MMM$ is a singleton, we say that $R_\MMM(\theta)$ \emph{lands}.

\section{The impression and real lamination of polynomials}\label{p-l}

Let $f_c$ be a quadratic polynomial with connected Julia set.
If $J_c$ is locally connected, the map $\phi_c^{-1}$ can be continuously extended to $\C\setminus\D$ and each external ray $R_c(t)$ lands at a  Julia point. The map $\g_c:\T\to J_f$, $t\mapsto \g_c(t)$, is continuous and surjective. In this case, the landing pattern of external rays for $f_c$ induces an equivalence relation $\lambda(c)$ on $\T$ such that $t\overset{\lambda(c)}{\thicksim} s$ if and only if $\g_c(t)=\g_c(s)$.

Kiwi \cite{Ki2} generalized the definition of such an equivalence relation to a class of non locally-connected  case, with the concept {\it impression} instead of the landing point in the locally-connected case.

We still assume that $f_c$ has the connected Julia set. Consider an argument $t\in \T$. We say that $z\in J_f$ belongs to the {\it impression} of $t$, written ${\rm Imp}_c(t)$, if and only if there exists a sequence  $\{z_n\in R_c(t_n)\}$ converging to $z$, with $\{t_n\}\subset \T$ converging to $t$. Note that $f_c(\text{Imp}_c(t))=\text{Imp}_c(\tau(t))$, and  ${\rm Imp}_c(t)=\g_c(t)$ for each $t\in \T$ if $J_c$ is locally connected.

Similar to the locally-connected case, we have the following two facts about the impressions, which will be used in the proof of Proposition \ref{part1}

\begin{lemma}\label{impression}
\begin{enumerate}
\item If there exists a sequence $\{z_n\in {\rm Imp}_c(t_n)\}$ with $z_n\to z$ and $t_n\to t$ as $n\to\infty$, then $z\in {\rm Imp}_c(t)$.
\item For any $z\in J_c$, there exists an argument $t\in \T$ with $z\in{\rm Imp}_c(t)$.
\end{enumerate}
\end{lemma}
\begin{proof}
1. For each $n$, since $z_n\in\text{Imp}_c(t_n)$, we can choose an argument $s_n$ and a point $w_n\in R_c(s_n)$ such that $|z_n-w_n|<|z_n-z|$ and $|t_n-s_n|<|t_n-t|$. Then we have $|w_n-z|<2|z_n-z|\to 0$ and $|s_n-t|<2|t_n-t|\to 0$ as $n\to \infty$. It means that $w_n\in R_c(s_n)$ converges to $z$ and $s_n$ converges to $t$, hence $z\in \text{Imp}_c(t)$.

\noindent 2. Let $z\in J_c$. Since $J_c$ is the boundary of the basin $\C\setminus K_c$, there exists a sequence $\{z_n\}\subset \C\setminus K_c$ with $z_n$ converging to $z$. Each $z_n$ belongs to an external ray of argument $t_n$. By picking a subsequence, we  assume $t_n\to t$ as $n\to \infty$. It follows from the definition that $z\in\text{Imp}_c(t)$.
\end{proof}

Let $f_c$ be a quadratic polynomial with connected Julia set and without irrational neutral cycles. Following Kiwi (see \cite[Definition 2.2]{Ki2}),
 the {\it real lamination} of $f_c$ is the smallest equivalence relation $\lambda(c)$ in $\T$ which identities $s$ and $t$ whenever ${\rm Imp}_c(s)\cap {\rm Imp}_c(t)\not=\emptyset$.  For a $\lambda(c)$-class $A$, we denote ${\rm Imp}_c(A)$ the union of the impressions ${\rm Imp}_c(t)$ for all $t\in A$.

\section{The equivalence relation generated by angles}\label{key-lemma}

For any angle $\theta\in\T$, its preimages $\{\theta/2,(\theta+1)/2\}$ under $\tau$ divide $\T$ into two closed half circles, which are denoted by $L_0,L_1$ respectively. Then we can endow each angle $t\in\T$ two itineraries $\iota^{\pm}_\theta(t)$ with respect to $\theta$ such that $\iota^{\pm}_\theta(t)= i_0i_1\ldots$ if, for each $n\geq0$, there exists $\epsilon>0$ with $(\tau^n(t),\tau^n(t)\pm\epsilon)\subset L_{i_n}$.

From the definition, we can see that if $t$ is not an iterated preimage of $\theta$, then $\iota_\theta^+(t)=\iota^-_\theta(t)$. In particular, we have $\iota^+_\theta(\theta)=\iota^-_\theta(\theta)$ for any non-periodic $\theta$. In this case, the sequence $\iota^+_\theta(\theta)=\iota^-_\theta(\theta)$ is called the {\it kneading sequence of $\theta$}, written $\nu(\theta)$. It is called {\it aperiodic} if it is not a periodic
symbol sequence under the shift map.

By Kiwi \cite[Definition 4.5]{Ki2}, the {\it equivalence relation generated by} $\theta\in\T$, denoted by $\lambda(\theta)$, is defined as the smallest equivalence relation
such that if $\iota_\theta^+(t)=\iota_\theta^-(s)$, then $s$ and $t$ are equivalent.

From now on, we always assume that $\theta\in\T$ is non-recurrent. The following is a key Lemma in our proof.

\begin{lemma}\label{aperiodic}
If $\theta$ is non-recurrent, then $\nu(\theta)$ is aperiodic.
\end{lemma}
\begin{proof}
On the contrary, we assume that $\nu(\theta)$ is periodic of period $p\geq1$. For each integer $0\leq k\leq p$, set
\[B_k:=\{\tau^{k+np}(\theta)\mid n\geq0\}.\]
Then we have $\tau(\ov{B_k})\subset\ov{B_{k+1}},k\in\{0,\ldots,p-1\},$ and $\ov{B_p}\subset \ov{B_0}$. Note that all elements of $B_k$ have a common itinerary, then each $B_k$ is contained in a component of $\T\setminus\{\theta/2,(\theta+1)/2\}$. Moreover, by the non-recurrent property, the closures $\ov{B_k}$ are disjoint from $\{\theta/2, (\theta+1)/2\}$. It follows that for each $k\in\{0,\ldots,p-1\}$, the map $\tau:\ov{B_k}\to \ov{B_{k+1}}$ is injective, and hence $\tau^p:\ov{B_0}\to \ov{B_0}$ is injective. According to \cite[Lemma 18.8]{Mi1}, the set $\ov{B_0}$ is finite, a contradiction.
\end{proof}

Since $\nu(\theta)$ is aperiodic, by \cite[Proposition 4.7]{Ki2}, we get that
\begin{proposition}\label{property}
 The equivalence relation $\lambda(\theta)$ is closed and satisfies that
\begin{enumerate}
\item  each $\lambda(\theta)$-class  is a finite subset of $\T$;
\item  if $A$ is a $\lambda(\theta)$-class, then $\tau(A)$ is a $\lambda(\theta)$-class;
\item  for any two different $\lambda(\theta)$-classes $A,B$, the convex hulls of $A$ and $B$ are disjoint.
\end{enumerate}
\end{proposition}
Combining the fact that $\theta$ is non-recurrent, we can obtain more information about $\lambda(\theta)$. Since $\theta$ is not periodic, then $\iota^+_\theta(\theta/2)=\iota^-_\theta((\theta+1)/2)$, and hence $\theta/2$ and $(\theta+1)/2$ are $\lambda(\theta)$-equivalent. We call the $\lambda(\theta)$-class containing $\{\theta/2,(\theta+1)/2\}$ the {\it critical class}, denoted by $C_\theta$; and the one containing $\theta$ the {\it characteristic class}, denoted by $A_\theta$.

\begin{lemma}\label{equivalence} Let $A_\theta$ be the characteristic class of $\lambda(\theta)$. Then we have
\begin{enumerate}
\item $A_\theta$ is \emph{wandering}, i.e., $\tau^n(A_\theta)\cap \tau^m(A_\theta)=\emptyset$ for each $0\leq m<n$;
\item  $A_\theta$ contains at most two angles;
\item  $A_\theta$ is non-recurrent, i.e,  $\exists~\delta>0$ s.t $dist(A_\theta,\tau^n(A_\theta))>\delta$ for all $n\geq 1$.
\end{enumerate}
\end{lemma}
\begin{proof}
{\it 1.} On the contrary, without loss of generality, we assume that $A_\theta$ is periodic. Then the fact of $\#A_\theta<\infty$ implies that $\theta$ is eventually periodic, a contradiction.

{\noindent \it 2.} Since $A_\theta$ is wandering, its orbit does not contain $C_\theta$. Note that each $\lambda(\theta)$-class except  $C_\theta$ is contained in one component of $\T\setminus \{\theta/2,(\theta+1)/2\}$ (by 3 of Proposition \ref{property}), it follows that $\#A_\theta=\#\tau^n(A_\theta)$ for all $n\geq0$. Using Thurston's No Wandering Polygon Theorem (\cite[Theorem II.5.2]{T}), the conclusion holds.

{\noindent \it 3.} If $A_\theta$ contains one angle, since $\theta$ is non-recurrent, the set $A_\theta$ is naturally non-recurrent. So, by assentation (2), we just need to prove the case that $\#A_\theta=2$.

For any $\alpha\not=\beta\in\T$, we define the arc $(\alpha,\beta)\subset \T$ as the closure of the connected component of $\T\setminus\{\alpha,\beta\}$ that consists of the angles we traverse if we move on $\T$ in the counterclockwise direction from $\alpha$ to $\beta$. The length of an arc $S\subset \T$ is denoted by $|S|$.
 We define a map $\sigma$ on all arcs in $\T$  such that $\sigma(\alpha,\beta)$ equals to $(\tau(\alpha),\tau(\beta))$ if $\tau(\alpha)\not=\tau(\beta)$, and equals to $\tau(\alpha)$ otherwise.
%The length $|S|$ of an arc $S=(\alpha,\beta)$ is defined as
%\[|S|=\left\{
 %       \begin{array}{ll}
  %        \beta-\alpha, & \hbox{if $\alpha<\beta$;} \\
   %       1-(\alpha-\beta), & \hbox{it $\alpha>\beta$.}
    %    \end{array}
    %  \right.
%\]
It is apparent that $|\sigma(S)|=2|S|$ if $|S|<1/2$ and $|\sigma(S)|=2|S|-1$ otherwise.

Let $A_\theta=\{\theta,\eta\}$.  Then it divides $\T$ into two closed arcs. We denote the shorter one by $S^+_1$ and the longer one by $S_1^-$. For $n\geq1$, set $$A_n:=\tau^{n-1}(A_\theta),\ S^+_n:=\sigma^{n-1}(S^+_1)\text{ and }S^-_n:=\sigma^{n-1}(S^-_1).$$
According to the proof of assentation 2, each $A_n$ is a $\lambda(\theta)$-class containing two angles. And it divides $\T$ into $S^+_n$ and $S^-_n$.
By 2 of Proposition \ref{property}, the critical class $C_\theta$ is equal to $\{\frac{\theta}{2},\frac{\theta+1}{2},\frac{\eta}{2},\frac{\eta+1}{2}\}$. It divides $\T$ into four arcs. We denote the two shorter ones by $S^+_0,-S^+_0$, and the two longer ones by $S^-_0,-S^-_0$. It is clear that $|S^+_0|=|-S^+_0|<|S^-_0|=|-S^-_0|<1/2$ and $\sigma(\pm S^\delta_0)=S^\delta_1$ for $\delta\in\{+,-\}$.

We claim that in the set $\{S^{\pm}_n\mid n\geq1\}$, the arc $S^+_1$ has the shortest length. If not, suppose that $k\geq2$ is the first integer such that $S^+_k$ or $S^-_k$, say $S^+_k$, has a shorter length than $S^+_1$. Then $|S_{k-1}^+|>1/2$. It follows from 3 of Proposition \ref{property} that the arc $S_{k-1}^+$ contains $C_\theta$, and hence contains three of the arcs $\pm S^+_0,\pm S^-_0$. Its image $\sigma(S_{k-1}^+)=S_k^+$ therefore contains either $\tau(S^+_0)=S^+_1$ or $\tau(S^-_0)=S^-_1$. It implies $|S_k^+|\geq\min\{|S_1^+|,|S_1^-|\}=|S_1^+|$, a contradiction to the assumption that $|S^+_k|<|S_1^+|$.

Set $\theta_n:=\tau^n(\theta)$ and $\eta_n:=\tau^n(\eta)$ for all $n\geq1$. We now start to prove point 3 by contradiction. We can assume that $\text{dist}(A_{n_k},A_\theta)\to 0$, $\theta_{n_k}\to\theta'$, and $\eta_{n_k}\to\eta'$ as $k\to\infty$ by passing to a subsequence if necessary. Since $\lambda(\theta)$ is closed, $\eta', \theta'$ are in the same $\lambda(\theta)$-class. Notice that $|S_{n_k}^{\pm}|\geq |S_1^+|$ as explained in the claim above. Then $\eta'\neq\theta'$. Since $\text{dist}(A_{n_k},A_\theta)\to 0$ and $\# A_\theta=2$, we have $\theta'=\eta$ and $\eta'=\theta$.

Hence for any $\epsilon>0$, there exist $k_1$ and $k_2$ such that $|\eta_{n_{k_1}}-\theta|<\epsilon/2$ and
$$2^{n_{k_1}}|\theta_{n_{k_2}}-\eta|=|\tau^{n_{k_1}}(\theta_{n_{k_2}})-\tau^{n_{k_1}}(\eta)|<\epsilon/2.$$
It follows that $|\tau^{n_{k_1}+n_{k_2}}(\theta)-\theta|<\epsilon$ for any $\epsilon>0$, which is impossible.
\end{proof}

\section{Proof of Theorem \ref{main}}\label{proving}

The proof of Theorem \ref{main} is based on the following two propositions.

%Note that non-recurrent quadratic polynomials are not infinite renormalization, then  Yoccoz Rigidity Theorem implies that
%\begin{proposition}\label{rigidity}
%If $f_c$ and $f_{c'}$ are non-recurrent quadratic polynomials with the same landing pattern of rational rays, i.e., $R_c(\alpha)$ and $R_c(\beta)$ land at a common point for $\alpha,\beta\in\Q/\Z$ if and only if $R_{c'}(\alpha)$ and $R_{c'}(\beta)$ land at a common point, then $c=c'$.
%\end{proposition}

\begin{proposition}\label{part1}
Given a non-recurrent angle $\theta$, the parameter ray $R_\MMM(\theta)$ lands at a non-recurrent parameter $c$ such that $f_c$ takes $\theta$ as a characteristic angle.
\end{proposition}
\begin{proof}
We set ${\rm Acc}_\MMM(\theta)$ the accumulation set of $R_\MMM(\theta)$ on $\MMM$. By Lemma \ref{aperiodic} and \cite[Theorem 1]{Ki2}, we have that  all cycles of $f_c$ are repelling and $\lambda(c)=\lambda(\theta)$ for all $c\in {\rm Acc}_\MMM(\theta)$. Let $c\in {\rm Acc}_\MMM(\theta)$. We will show that $f_c$ is non-recurrent with a characteristic angle $\theta$.

By Lemma \ref{impression} (2), we choose an angle $t_0\in \T$ with $0\in {\rm Imp}_c(t_0)$. Note that the B\"{o}ettcher coordinate $\phi_c$ satisfies that $\phi_c(-z)=-\phi_c(z)$ for $z\in \C\setminus K_c$, then the sets ${\rm Imp}_c(t_0)$ and ${\rm Imp}_c(t_0+1/2)$ are symmetric about the origin. It follows that $0\in {\rm Imp}_c(t_0)\cap {\rm Imp}_c(t_0+1/2)$, and hence $t_0,t_0+1/2$ are in a common $\lambda(\theta)$-class (because $\lambda(c)=\lambda(\theta)$). By Proposition \ref{property} (3), we know that $t_0,t_0+1/2$ are contained in the critical class $C_\theta$.
Then the impression of the characteristic class ${\rm Imp}_c(A_\theta)$ contains the critical value $c$.

Set $c_n:=f^{n-1}_c(c)$ and recall that $A_n:=\tau^{n-1}(A_\theta)$. Then $c_n$ belongs to ${\rm Imp}_c(A_n)$ for each $n\geq1$. The non-recurrent property of $f_c$ is equivalent to that the accumulation set of $\{c_n,n\geq1\}$ is disjoint from the critical value $c$. We continue the argument by contradiction and assume that the sequence $\{c_{n_k}\in \text{Imp}_c(t_{n_k})\}$ with $t_{n_k}\in A_{n_k}$ converges to $c$ as $n\to \infty$.

By passing to a subsequence if necessary, we  assume that $t_{n_k}\to t$ as $k\to\infty$. From Lemma \ref{impression}(1), we see that $c\in{\rm Imp}_c(t)$, and hence $t\in A_\theta$. This contradicts Lemma \ref{equivalence} (3). Thus $f_c$ is non-recurrent.
Since the Julia sets of non-recurrent quadratic polynomials are locally connected,   the set ${\rm Imp}_c(A_\theta)$ reduces to one point $c$. So $R_c(\theta)$ lands at $c$.

We have seen that all $f_c$ with $c\in {\rm Acc}_\MMM(\theta)$ are non-recurrent and have a common  real lamination $\lambda(\theta)$. Due to Yoccoz Rigidity Theorem \cite[Theorem III]{H} or  \cite[Theorem 4.1]{Ze}, we have that ${\rm Acc}_\MMM(\theta)$ reduces to one point. So the parameter ray $R_\MMM(\theta)$ lands.
\end{proof}

\begin{proposition}\label{part2}
Let $f_c$ be a non-recurrent quadratic polynomial. Then it has at most two characteristic angles, and the parameter rays at these angles land at $c$.
\end{proposition}
\begin{proof}
Let $f_c$ be a non-recurrent quadratic polynomial. Since $J_c$ is locally connected, there exists a characteristic angle $\theta$ with $R_c(\theta)$ landing at $c$. Clearly $\theta$ is non-recurrent. By Proposition \ref{part1}, the parameter ray $R_\MMM(\theta)$ lands at a non-recurrent parameter $c'$ so that $f_{c'}$ has a characteristic angle $\theta$. We then have that $\theta/2$ and $(\theta+1)/2$ are contained in both a $\lambda(c)$-class and a $\lambda(c')$-class. By \cite[Proposition 4.10]{Ki2}, we get $\lambda(c)=\lambda(c')=\lambda(\theta)$. Using again the Yoccoz Rigidity Theorem, it follows that $c=c'$. By Lemma \ref{equivalence} (2), the cardinality of the $\lambda(c)$-class that contains $\theta$ is at most two. Hence $f_c$ has at most two characteristic angles.
\end{proof}

\begin{proof}[Proof of Theorem \ref{main}]
It follows directly from Propositions \ref{part1} and \ref{part2}.
\end{proof}

\vspace{1cm}
\noindent Yan Gao, \\
Mathematical School  of Sichuan University, Chengdu 610064,
P. R. China. \\
Email: gyan@scu.edu.cn
\vspace{0.2cm}

\noindent Jinsong Zeng, \\
Academy of Mathematics and Systems Science, \\
Chinese
Academy of Sciences, Beijing 100190, P. R. China. \\
Email: zeng.jinsong@amss.ac.cn
\vspace{0.2cm}

\end{document}